[1994/12/01]
\documentclass[reqno,12pt]{amsart}
\usepackage{bbm}
\def\1{\mathbbm{1}}

\usepackage{amsmath,amsthm}
\usepackage[mathscr]{euscript}
\usepackage{mathrsfs}
\usepackage{graphicx}
\usepackage{amsmath,amssymb}

\def\t{\tau}

\def\calQ{{\mathcal{Q}}}
\def\calM{{\mathcal{M}}}

\def\calU{{\mathcal{U}}}

\def\calT{{\mathcal{T}}}

\def\calB{{\mathcal{B}}}
\def\calG{{\mathcal{G}}}
\def\calA{{\mathcal{A}}}

\def\calC{{\mathcal{C}}}

\def\calX{{\mathcal{X}}}

\def\calK{{\mathcal{K}}}
\def\calD{{\mathcal{D}}}

\def\calP{{\mathcal{P}}}

\def\C{\mathbb C}

\def\N{\mathbb N}

\newtheorem{definition}{Definition}[section]
\newtheorem{lemma}{Lemma}[section]
\newtheorem{proposition}{Proposition}[section]
\newtheorem{corollary}{Corollary}[section]
\newtheorem{theorem}{Theorem}[section]
\theoremstyle{remark}
\newtheorem{remark}{Remark}[section]
\newtheorem{example}{Example}[section]

\title{Well-poesdness and approximate controllability of neutral network systems}
\author{Y. El Gantouh and S. Hadd}
\address{Department of Mathematics, Faculty of Sciences, Ibn Zohr University, Hay Dakhla, BP8106, 80000--Agadir, Morocco; elgantouhyassine@gmail.com, s.hadd@uiz.ac.ma}
\thanks{This work has been supported by COST Action CA18232. The authors would like to thank Prof. A. Rhandi for discussing some part of this paper}

\begin{document}
\maketitle

\renewcommand{\sectionmark}[1]{}
\begin{abstract}
In this paper, we study the concept of approximate controllability of retarded network systems of neutral type. In one hand, we reformulate such systems as free-delay boundary control systems on product spaces. On the other hand, we use the riche theory of infinite-dimensional linear systems to derive necessary and sufficient conditions for the approximate controllability. Moreover, we propose a rank condition for which we can easily verify the conditions of controllability. Our approach is mainly based on feedback theory of regular linear systems in the Salamon-Weiss sense.
 \end{abstract}

\section{Introduction}\label{intro}
The main object of this paper is to characterise the approximate controllability of the following retarded network system of neutral type and input delays
\begin{align}\label{SY1}
\begin{cases}
     \dfrac{\partial }{\partial t}\varrho^j(t,x)
    =c^j(x)\dfrac{\partial }{\partial x}\varrho^j(t,x)+q^j(x)\varrho^j(t,x) +\displaystyle\sum_{k=1}^{m}L_{jk}z^{k}(t+\cdot,\cdot),\; x\in (0,1), t\geq 0,\\
\varrho^j(0,x)= g^{j}(x),\;\; x\in (0,1), \\
      \mathsf{i}^{-}_{ij}c^{j}(1)\varrho^j(t,1)=
 \mathsf{w}_{ij}^{-}\displaystyle\sum_{k=1}^{m} \mathsf{i}^{+}_{ik}c^{k}(0)\varrho^j(t,0)
 +\sum_{l=1}^{n_0}\mathrm{k}_{il}v^{l}(t),\qquad\qquad\qquad t\geq 0,\\
      z^j(\theta,x)=\varphi^{j}(\theta,x),\; u^{j}(\theta)=\psi^{j}(\theta) ,\qquad\qquad\quad\qquad\qquad\qquad\;\;\theta\in[-r,0],x\in (0,1),\\
      \varrho^j(t,x)=\left[z^{j}(t,x)- \displaystyle\sum_{k=1}^{m}D_{jk}z^{k}(t+\cdot,\cdot)-\displaystyle\sum_{i=1}^{n}\mathsf{k}_{ij}u^{j}(t+\cdot)-\mathrm{b}_{ij}u^{j}(t)\right]
     \end{cases}
\end{align}
for $ i=1,\ldots,n$ and $ j=1,\ldots,m$. Here, $ z^j(t,x) $ represents the distribution of the material along an edge $ e_{j} $ of a graph $ \mathsf{G} $ at the point $ x $ and time $ t $, where $ \mathsf{G} $ is a finite connected graph composed by $ n \in \mathbb{N}$ vertices $ \alpha_{1}\,,\ldots,\alpha_n $, and by $ m\in \mathbb{N} $ edges $ e_{1},\,\ldots,e_m $ which are assumed to be normalized on the interval $ [0,1] $. Moreover, we assume that the nodes $\alpha_{1},\,\ldots,\alpha_n $ exhibit standard Kirchhoff type conditions to be specified later on. Moreover, $z^j(t+\cdot,x):[-r,0]\to \mathbb{C}$ is the history function of $z^j$, and $u^j(t+\cdot):[-r,0]\to \mathbb{C}$ is the history function of the control function $u^j$. We notice that system \eqref{SY1} arises as a model for linear flows in networks with fading memory.

The study of system \eqref{SY1} is motivated by the several open problems on transport network systems, which is a very active topic for many years \cite{BaPu2,BDK,BDR12,BKNR,EK,EKKNS,EHR1}. Such research activity is motivated by a broad area of their possible applications,  see, e.g. \cite{BLMCH}, and the interesting mathematical questions that arise from their analysis. For instance, several properties of the transport processes depend on the structure of the network and on the rational relations of the flow velocities, see, e.g. \cite{BaPu1,TE} and references therein.


On the other hand, neutral delay systems arise naturally in many practical mathematical models. Typical examples include communication networks, structured population models, chemical processes, tele-operation systems \cite{HV,Zhong}. The qualitative properties (existence, stability, controllability, etc.) for this class of systems have received much attention (see \cite{BP}, \cite{BH}, \cite{HNN}, \cite{HV}, \cite{Sa1}, \cite{Zhong} and references therein). For instance, different controllability results for various neutral delay systems have been established recently (see, \cite{RaS}, \cite{DKR}, \cite{SNU}, \cite{STH}). In \cite{RaS}, the authors analyze the exact null controllability of neutral systems with distributed state delay by using the moment problem approach. In \cite{DKR}, relative controllability of linear discrete systems with a single constant delay was studied using the so-called discrete delayed matrix exponential. In \cite{SNU}, the authors studied the approximate controllability of linear (continuous-time) systems with state delays via the matrix Lambert $W$ function. The robustness of approximate controllability of linear retarded systems under structured perturbations has been addressed in \cite{STH} using the so-called structured distance to non-surjectivity. However, the results established in the aforementioned works become invalid for the transport network system \eqref{SY1}, since the operators $ D=(D_{jk}), L=(L_{jk}), K_1=(\mathsf{k}_{jk}) $ and $ A_m $ are supposed to be unbounded. In fact, as we shall see in Section \ref{sec:5}, if we take $ X=L^{p}([0,1])^{m} $ as the state space then $ D,L \in\mathcal{L}(W^{1,p}([-r,0],X) ; X) $ and $ K_1\in\mathcal{L}(  W^{1,p}([-r,0],\mathbf{C}^{n});X) $.

In this paper, we study the concept of approximate controllability of boundary value problems of neutral type with a particular aim to explore new techniques and new questions for control problems of transport network systems. We formulate the problem in the framework of well-posed and regular linear systems and solve it in the operator form. To be precise, we use product spaces and operator matrices to reformulate \eqref{Sy.1} into a inhomogeneous perturbed Cauchy problem governed by an operator having a perturbed domain. This allows us to use the feedback theory of well-posed and regular linear systems to prove that this operator is a generator. Our approach allows us to easily calculate the spectrum and the resolvent operator of this generator. In this manner, necessary and sufficient conditions of approximate controllability for \eqref{Sy.1} are formulated and proved by using the feedback theory of regular linear systems and methods of functional analysis. Our main result is that, when the control space is of finite dimension, we prove that the established approximate controllability criteria are reduced to a compact rank condition given in terms of transfer functions of controlled delay systems. As we shall see in Section \ref{sec:3} our approach by transforming the neutral delay system controllability problem into approximate controllability of an abstract perturbed boundary control problem greatly facilitates analysis and offers an alternative approach for the study of controllability in terms of extensive existing knowledge of feedback theory of closed-loop systems. This establishes a framework for investigating the approximate controllability of infinite dimensional neutral delay systems with state and input delays, which may shed some light in solving the approximate controllability of concrete physical problems.

The whole article is organized as follows: we initially present a survey on well-posed and regular linear systems in the Salamon-Weiss sense; Section \ref{sec:1}. The results obtained on the well-posedness and spectral theory of boundary value problems of neutral type are discussed in Section \ref{sec:2}. Section \ref{sec:3} is devoted to state and prove the main results on approximate controllability of abstract boundary control systems of neutral type.  Finally, in Section \ref{sec:5}, we show the solvability of transport network systems of neutral type by means of our introduced framework.

\section{Some background on infinite-dimensional linear systems}
\label{sec:1}
In this section we recall some well-known results and definitions on infinite dimensional linear time-invariant systems. The reader is referred to the papers  \cite{Sa1}, \cite{WO}, \cite{WC}, \cite{WR}, \cite{WT}, which was our main reference, if more details or further references are required. For the Hilbert space or Banach space setting, the reader may also refer to \cite{Sta,TW}.

Let $ X,U,Z $ be Banach spaces such that $ Z \subset X $ with continuous dense embedding, $ A_{m}:Z \longrightarrow X $ be a closed linear (often differential) operator on $ X $ (here $D(A_{m})=Z$), and boundary linear operators $ G,M: Z\subset X\longrightarrow U$.

Consider the following boundary input-output system
\begin{align}\label{S2.Sy1}
\begin{cases}
     \dot{z}(t)=A_{m} z(t),& t\ge 0,\; z(0)=z_{0}\\
     G z(t)= u(t),& t\ge 0,\\
      y(t) = M z(t), & t\ge 0.
     \end{cases}
 \end{align}
Notice that the well-posed of the boundary input-output system \eqref{S2.Sy1} consists in finding conditions on operators $A_m, G$ and $M$ such as
\begin{align}\label{S2.1}
\Vert z(\tau)\Vert^{p}_{X}+\Vert y(\cdot)\Vert^{p}_{ L^{p}([0,\tau];U)}\leq c(\tau)\left( \Vert z_{0}\Vert^{p}_{X}+\Vert u(\cdot)\Vert^{p}_{ L^{p}([0,\tau];U)}\right).
\end{align}
for some (hence for every) $ \tau>0 $, a constant $ c(\tau)>0 $ and $ p\geq 1 $. To make these statements more clear, some hypothesis are needed.
\begin{itemize}
\item[\textbf{(A1)}] the restricted operator $ A\subset A_m $ with domain $ D(A)=\ker G $ geenrates a C$_{0} $-semigroup $ (T(t))_{t\geq 0} $ on $ X $;
\item[\textbf{(A2)}]  the boundary operator $ G $ is surjective;
\end{itemize}
According to assumptions \textbf{(A1)} and \textbf{(A2)}, for $ \mu \in\rho (A) $, the following inverse, called the Dirichlet operator,
$$
D_\mu=(G_{ \vert_{\ker ( \mu -A_{m} )}})^{-1}\in \mathcal{L}(U, D(A_m)).
$$
exists. Define the boundary control operator
$$
B= (\mu -A_{-1})D_\mu\in \mathcal{L}(U,X_{-1}),
$$
then $ B\in\mathcal{L}(U,X_{-1}) $, $ \text{Rang } B \cap X=\lbrace 0\rbrace $ and
\begin{align}\label{S2.2}
(A-A_{-1})_{\vert_{Z}}=BG,
\end{align}
since $\mu D_\mu u=A_mD_\mu u,\,u\in U$, where $ X_{-1} $ is the extrapolation space associated with $ X $ and $ A $, see for more details \cite{EN}. We mention that the operator $ B $ is independent of $ \mu $ due to the resolvent equation. By virtue of formula \eqref{S2.2} the boundary input-output system \eqref{S2.Sy1} can be reformulated as the following distributed-parameter system
\begin{align}\label{S2.Sy2}
\begin{cases}
     \dot{z}(t)= A_{-1} z(t)+Bu(t),& t\geq 0,\; z(0)=z_{0},\\
      y(t) = C z(t),&t\geq 0,
\end{cases}
 \end{align}
where $$C=M_{\vert_{D(A)}}.$$
Then the state of the system \eqref{S2.Sy2} satisfy the variation of constants formula
\begin{align}\label{S2.4}
z(t;z_{0},u)=T(t)z_{0}+\int_{0}^{t}T_{-1}(t-s)Bu(s)ds, \; \; t\geq 0,
\end{align}
for all $ z_0\in X $ and $ u\in L^{p}([0,+\infty);U) $. Notice that the integral in \eqref{S2.4} is taken in the large space $ X_{-1} $. Thus, we need a class of control operators $ B $ for which the state of the system \eqref{S2.Sy2} takes values in the state space $X$. This motivated the following definition.
\begin{definition}\label{D1.S2}
An operator $B\in\mathcal{L}(U,X_{-1})$ is called an admissible control operator for $A$, if for some $\tau >0 $
\begin{align*}
\Phi_{\tau}u:=\int_{0}^{\tau}T_{-1}(\tau-s)Bu(s)ds,
\end{align*}
takes values in $X$ for any $ u\in L^{p}([0,+\infty);U)  $.
\end{definition}
Note that the admissibility of $ B $ implies that the state of the system \eqref{S2.Sy2} is a continuous $ X $-valued function of $ t $ and satisfy
\begin{align}\label{S2.5}
z(t)=T(t)z_{0}+\Phi_{t}u,
\end{align}
for all $ z_{0}\in X $ and $ u\in L^{p}_{loc}([0,\infty];U) $. Moreover, if $u \in L^{p}_{\alpha}([0,\infty);U) $ for some $ \alpha>w_{0}(T) $ (where $ L^{p}_{\alpha}([0,\infty),U) $ denote the space of all the functions of the form $ u(t)=e^{\alpha t}v(t) $, where $ v\in L^{p}([0,\infty),U) $). Then $ u $ and $ z $ from \eqref{S2.5} have Laplace transforms related by
\begin{align}\label{S2.6}
\hat{z}(\mu)=R(\mu,A)z_0+\widehat{\Phi_\bullet u}(\mu),\;\;\; \text{ with }\;\;\; \widehat{\Phi_\bullet u}(\mu)=D_\mu \hat{u}(\mu),\;\; \forall\; \Re e\,\mu> \alpha,
\end{align}
where $ \alpha\in \mathbb{R} $ and $\hat{u}$ denote the Laplace transform of $u$.

On the other hand, if the control function $u$ is smooth enough, let say
$$ u\in W^{2,p}_{0,loc}([0,+\infty),U):=\left\{u\in W^{2,p}_{loc}([0,+\infty),U):u(0)=u'(0)=0\right\} .$$
It follows that $ \Phi_{t}u\in Z$ for any $t\ge 0$. It makes sense to define the linear operator
\begin{align}\label{S2.7}
(\mathbb{F} u)(t)=M\Phi_t u,\;\; t\geq 0, \;\;  u\in W^{2,p}_{loc}([0,\infty);U).
\end{align}
With these notations, it follows that
$$y(t)=\Psi z_{0}+\mathbb{F} u, \;\; \text{ for any } (z_0,u)\in D(A)\times W^{2,p}_{loc}([0,\infty);U),  $$
where $$ \Psi z_0=CT(.)z_0 ,\;\; z_0\in D(A) .$$
 So, according to formula \eqref{S2.1}, we are seeking for an output function in $ L^{p}_{loc}([0,\infty];U) $ for any  $ (z_0,u) \in X \times L^{p}_{loc}([0,\infty];U) $.  As a matter of fact, the previous property may not hold for unbounded operator $ C $. In order to overcome this obstacle we first define the following class of operators $ C $.
\begin{definition}\label{S2.D2}
An operator $ C\in\mathcal{L}(D(A),U)$ is called an admissible observation operator for $ A $ if
\begin{align}\label{S2.8}
\int_{0}^{\tau} \Vert CT(s)z\Vert^{p}ds\leq \gamma^{p}\Vert z\Vert^{p},
\end{align}
for all $ z\in D(A) $, $ p\in[1,\infty) $ and a constant $ \gamma:=\gamma (\tau)> 0 $ with $ \tau\geq 0 $.
\end{definition}
In particular, for an admissible observation operator $ C $ the map $\Psi$ from $D(A)$ to $ L^{p}_{loc}([0,\infty);U)$ can be extended to a linear bounded operator $ \Psi:X\longrightarrow L^{p}_{loc}([0,\infty);U) $. Moreover, as shown by Weiss \cite{WO}, one can associate with the operator $ C\in\mathcal{L}(D(A),U)$ the following operator
\begin{align*}
 C_{\Lambda}z:=\lim_{\mu \longrightarrow \infty} C\mu  R(\mu ,A)z,
\end{align*}
whose domain $D(C_{\Lambda})$ consists of all $ z_0\in X $ for which the limit exists, called the $ \Lambda $-extension of $ C $ for $ A $. The introduced operator
 makes possible to give a simple pointwise interpretation of the output map $ \Psi $ in terms of the observation operator $ C $. Moreover, for an admissible observation operator $C $ we have $ T(t)z\in D(C_{\Lambda}) $ for a.e $ t\geq 0 $ and
$$
\Psi z:= C_{\Lambda}T(.)z.
$$
\begin{definition}\label{S2.D3}
Let $B \in\mathcal{L}(U,X_{-1})$ and $C\in \mathcal{L}(D(A),U)$ be admissible control and observation operator for $ A $, respectively. We call the triplet $ (A,B,C ) $ (or equivalently the system \eqref{S2.Sy2}) a well-posed state-space operators on $ U, X, U $, if for every $ \tau >0 $ there exits $ \kappa=\kappa(\tau) $ such as
\begin{align*}
\Vert \mathbb{F} u\Vert_{L^{p}([0,\tau],U)}\leq \kappa \Vert u\Vert_{L^{p}([0,\tau],U)},\qquad u\in W^{2,p}_{0,loc}([0,+\infty),U).
\end{align*}
\end{definition}
Now, for $ \tau\geq 0 $, we define the input-output maps of $ (A,B,C) $, denoted by $ \mathbb{F}_{\tau} $, by truncating the output to $ [0,\tau] $:
\begin{align*}
\mathbb{F}_{\tau}u=(\mathbb{F}u)_{\vert_{[0,\tau]}}.
\end{align*}
In particular, the feedback law $ u=y $ has a sense if only if $ (I-\mathbb{F})u=\Psi z_0 $ has a unique solution $ u\in L^{p}([0,\tau],U) $ for some $ \tau>0 $. This is true if $ I-\mathbb{F} $ is invertible in $ L^{p}([0,\tau],U) $. In this case, the identity $ I:U\longrightarrow U $ is called an admissible feedback for $ (A,B,C) $.

A more appropriate subclass of well-posed state-space operators is defined by:
\begin{definition}\label{S2.D4}
Let $ (A,B,C) $ a well-posed state-space operators on $ U,X,U $. Then, the triplet $ (A,B,C) $ is called regular state-space operators (with feedthrough zero) if for any $ v\in U $, we have
\begin{align*}
\lim_{\tau\longmapsto 0}\int_{0}^{\tau}(\mathbb{F}(\1_{\mathbb{R}_{+}} \cdot v))(\sigma)d\sigma=0.
\end{align*}

\end{definition}
Accordingly, the operator
$$
 A^{K}=A_{-1}+BKC_{\Lambda},\; \; D(A^{K})=\{z\in D(C_{\Lambda}): (A_{-1}+BKC_{\Lambda})z\in X\}
$$
generates a strongly continuous semigroup $T^{K}:=(T^{K}(t))_{t\geq 0}$ on $ X $ such that $ T^{K}(t)z \in D(C_{\Lambda}) $ for all $ z\in X $ and a.e $t\geq 0  $. Moreover, we have
$$
  T^{K}(t)z=T(t)z+\int_{0}^{t}T_{-1}(t-s)BKC_{\Lambda}T^{K}(s)zds
$$
 for all $ z\in X $ and $t\geq 0  $. For more details and references (see e.g. \cite{WR} and \cite[Chap.7]{Sta}).

\section{Wellposedness of boundary value problems of neutral type}\label{sec:2}
In this section, we investigate the wellposedness of the abstract boundary control systems of neutral type described as
\begin{align}\label{Sy.1}
\begin{cases}
\dfrac{d}{dt}(z(t)-Dz_{t}-K_{0}u(t)-K_{1}u_{t})\\
=A_m(z(t)-Dz_{t}-K_{0}u(t)-K_{1}u_{t})+Lz_{t}+B_{0}u(t)+B_{1}u_{t},& t\geq 0,\\
\underset{t\longrightarrow 0}{\lim}(z(t)-Dz_{t}-K_{0}u(t)-K_{1}u_{t})=\varrho_{0},\\
G(z(t)-Dz_{t}-K_{0}u(t)-K_{1}u_{t})
=M(z(t)-Dz_{t}-K_{0}u(t)-K_{1}u_{t})+Kv(t),& t\geq 0, \\
z_{0}=\varphi,\quad u_{0}=\psi,
\end{cases}
\end{align}
where the state variable $ z(.) $ takes values in a Banach space $ X $ and the control functions $ u(.),v(.) $ are given in the Banach space $ L^{p}_{loc}([0,\infty);U)$, where $ U $ is also a Banach space. $ K_{0},B_{0}$ are bounded linear operator from $ U $ to $ X $, whereas $ K$ is a boundary control operator from $ U $ to the Banach space $ \partial X $. $A_m:D(A_m)\subset X\longrightarrow X$ is a closed, linear differential operator and $ G,M: D(A_m)\longrightarrow\partial X$ are unbounded trace operators. The delay operators $ D,L: W^{1,p}([-r,0],X) \longrightarrow X$ and $ K_{1},B_{1}: W^{1,p}([-r,0],U)\longrightarrow X $ are defined by
\begin{align*}
D\varphi &=\int_{-r}^{0}d\eta(\theta)\varphi(\theta),& L\varphi&=\int_{-r}^{0}d\gamma(\theta)\varphi(\theta),\\
 B_{1}\psi &=\int_{-r}^{0}d\nu(\theta)\psi(\theta),& K_{1}\psi&=\int_{-r}^{0}d\vartheta(\theta)\psi(\theta),
\end{align*}
for $ \varphi\in W^{1,p}([-r,0],X) $ and $ \psi\in W^{1,p}([-r,0],U) $, where $ \eta,\gamma:[-r,0]\longrightarrow \mathcal{L}(X) $ and $\nu,\vartheta:[-r,0]\longrightarrow \mathcal{L}(U,X) $ are functions of bounded variations with total variations $ \vert\eta\vert([-\varepsilon,0])$, $\vert\gamma\vert([-\varepsilon,0]) $, $ \vert\nu\vert([-\varepsilon,0]), \text{ and }\vert\vartheta\vert([-\varepsilon,0])$ approach $ 0 $ as $ \varepsilon\longrightarrow 0 $. The notation $ x_{t} $ (resp. $ u_{t} $) represents the history function defined by $ x_{t}:[-r,0]\longrightarrow X $, $ x_{t}(\theta)=x(t+\theta) $ (resp. $ u_{t}(\theta)=u(t+\theta) $) for $ t\geq 0 $ and $ \theta\in [-r,0]  $ with delay $ r> 0 $. The functions $ z_{0}=\varphi, u_{0}=\psi $ are the initial history functions of $ z(.) $ and $ u(.) $, respectively. Some new notation is needed.
Let $ E $ be a Banach sapce and $ r>0 $ be a real number. Define the operator
$$
Q_{m}^{E}\xi = \frac{\partial}{\partial \theta}\xi, \;\;  D(Q_{m}^{E})= W^{1,p}([-r,0],E),
$$
and
$$
Q^{E}\xi = \frac{\partial}{\partial \theta}\xi, \;\;  D(Q^{E})= \lbrace \xi\in W^{1,p}([-r,0],E): \xi (0)=0\rbrace.
$$
It is well known that $ (Q^{E},D(Q^{E})) $ generate the left shift semigroup
$$
 (S^{E}(t)\xi)(\theta)=
\begin{cases}
0,\qquad \quad t+\theta\geq 0,
\\
\xi(t+\theta),\;t+\theta\leq 0,
\end{cases}
$$
for $ t\geq 0 $ and $ \theta\in [-r,0] $ and $ \xi\in L^{p}([-r,0],E) $; see \cite{EN}. For $ \mu \in\mathbf{C} $, we define the operator $ e_{\mu} $ as
\begin{align*}
e_{\mu}^{E}: E\longrightarrow L^{p}([-r,0],E),\;\; (e_{\mu}^{E}z)(\theta)=e^{\mu\theta}z,\;\;  z\in E,\; \theta\in [-r,0].
\end{align*}
Denote by $ Q_{-1}^{E} $ the extension of $ Q^{E} $ in the extrapolation sense and define the operator
\begin{align*}
\beta^{E}:= -Q_{-1}^{E}e_{0}.
\end{align*}
Then the function $  z_t(.) $ is the solution of the following boundary equation
\begin{align}\label{Sy.shift}
\left\lbrace
\begin{array}{lll}
\dot{v}(t,\theta)= Q^{E}_m v(t,\theta),\quad\qquad\quad t> 0,\;  \theta\in [-r,0],  \\
v(t,0)= z(t),\; \;  v(0,.)=\xi, \quad t\geq 0.
\end{array}
\right.
\end{align}
Moreover, for any $ \xi\in L^{p}([-r,0];E) $ and $ z\in L^{p}([-r,\infty);E) $ with $ z_0=\xi $, $ z_t $ is given by
\begin{align*}
z_t=S^{E}(t)\xi+\Phi^{E}_t z,\; \; t\geq 0,
\end{align*}
where $ \Phi_{t}: L^{p}([0,\infty);E) \longrightarrow  L^{p}([-r,0];E) $ are the linear operators defined by
$$
 (\Phi^{E}(t)z)(\theta)=
\begin{cases}
z(t+\theta),\;t+\theta\geq 0,
\\
0,\qquad \quad  t+\theta\leq 0,
\end{cases}
$$
for $ t\geq 0 $, $ z\in L^{p}([0,\infty),E) $ and $ \theta\in [-r,0] $; see \cite{HIR}.

Next we focus our attention to study the well-posedness of the neutral delay system \eqref{Sy.1} with no control input (i.e., $ B_0\equiv 0,K_0\equiv 0$ and $K\equiv 0 $).  For that purpose we introduce the Banach spaces
\begin{align*}
\mathcal{X} &:=X\times  L^{p}([-r,0],X)\times  L^{p}([-r,0],U),\\
\mathcal{Z} &:= D(A_m)\times W^{1,p}([-r,0],X)\times  W^{1,p}([-r,0],U),\\
\mathcal{U}&:=\partial X \times X\times U,
\end{align*}
equipped with their usual norms. Moreover, if we set
$$ \varrho(t)=z(t)-Dz_{t}-K_{0}u(t)-K_{1}u_{t}, $$
and using \eqref{Sy.shift} together with the function
\begin{align*}
t\longmapsto \zeta(t)=\left(\begin{smallmatrix}
z(t)\\
z_t\\
u_t
\end{smallmatrix}\right)
\end{align*}
one can see that the neutral delay system \eqref{Sy.1} can be rewrite as the following perturbed Cauchy problem
\begin{align}\label{S3.1.Sy1}
\begin{cases}
\dot{\zeta}(t)=[\mathcal{A}_{G,M}+\mathcal{P}]\zeta(t),&  t\geq 0,\\
\zeta(0)=(\varrho_0,\varphi,\psi)^{\top},&
\end{cases}
\end{align}
where $ \mathcal{A}_{G,M} $ and $ \mathcal{P} $ are linear operators on $ \mathcal{X} $ defined by
\begin{align*}
\mathcal{A}_{G,M}&:=\left(\begin{smallmatrix}
A_m & L & B_1\\
0 & Q_{m}^{X} & 0\\
0 & 0 & Q_{m}^{U}
\end{smallmatrix}\right),& D(\mathcal{A}_{G,M })&:=\left\{\left(\begin{smallmatrix}\varrho_0\\ \varphi\\ \psi\end{smallmatrix}\right)\in \mathcal{Z} :\; \begin{smallmatrix}
 G\varrho_0=M\varrho_0\\ \varphi(0)=\varrho_0+D\varphi+K_1 \psi
 \end{smallmatrix}\right\}\\
 \mathcal{P}&:=\left(\begin{smallmatrix}
0 & L & B_1\\
0 & 0 & 0\\
0 & 0 & 0
\end{smallmatrix}\right), & D(\mathcal{P})&:=\mathcal{Z}.
\end{align*}

We focus our attention now to study the well-posedness of the perturbed Cauchy problem \eqref{S3.1.Sy1}. To do this, we first claim that the operator $ \mathcal{A}_{G,M} $ is a generator on $ \mathcal{X} $. Second, we show that $ \mathcal{P} $ is a Miyadera-Voigt perturbation for $ \mathcal{A}_{G,M} $. Let assume that the boundary operator $ G $ satisfies the assumptions \textbf{(A1)} and \textbf{(A2)} (see Section \ref{sec:1}). On the other hand, we set
$$
C=M_{\vert_{D(A)}}.
$$
We also assume that
\begin{itemize}
\item[\textbf{(A3)}] the triple operator $ (A,B,C) $ is a regular state-space operators on $ \partial X,X, \partial X $ with the identity operator $ I_{\partial X} $ as admissible feedback operator.
\end{itemize}

We have the following result.
\begin{proposition}\label{S3.P1}
Under the assumptions \textbf{(A1)}-\textbf{(A3)}, the operator $ ( \mathcal{A}_{G,M},D(\mathcal{A}_{G,M})) $ generates a strongly continuous semigroup
$ (\mathcal{T}_{G,M}(t))_{t\geq 0} $ on $ \mathcal{X} $.
\end{proposition}

\begin{proof}
To prove our claim we shall use \cite[Theorem 4.1]{HMR}. To this end, we define the operators $ \mathcal{G},\mathcal{M}:\mathcal{Z}\longrightarrow \mathcal{ U}$ by
$$ \mathcal{G}=\left(\begin{smallmatrix}
G & 0 & 0\\
0 & \delta_{0} & 0\\
0 & 0 & \delta_{0}
\end{smallmatrix}\right),\;\;\mathcal{M}=\left(\begin{smallmatrix}
M & 0 & 0 \\
I & D & K_{1}\\
0 & 0 & 0
\end{smallmatrix}\right).$$
Then, we can rewrite the domain of $ \mathcal{A}_{G,M} $ as
\begin{align*}
D(\mathcal{A}_{G,M })&:=\left\{\left(\begin{smallmatrix}\varrho_0\\ \varphi\\ \psi\end{smallmatrix}\right)\in \mathcal{Z} :\; \mathcal{G}\left(\begin{smallmatrix}\varrho_0\\ \varphi\\ \psi\end{smallmatrix}\right)=\mathcal{M}\left(\begin{smallmatrix}\varrho_0\\ \varphi\\ \psi\end{smallmatrix}\right)\right\}
\end{align*}
Clearly, by virtue of the assumptions \textbf{(A1)} and \textbf{(A2)}, the operator $ \mathcal{G} $ is surjective and the operator $\mathcal{A}:=\mathcal{A}_{m}$ with $ D(\mathcal{A}):= \ker \mathcal{G}$ generates a diagonal C$_{0}$-semigroup $ (\mathcal{T}(t))_{t\geq 0} $ on $ \mathcal{X} $. On the other hand, if $ D_\mu $ ( $ \mu \in \rho(A)) $, then
\begin{align}\label{Bneu}
\mathcal{D}_{\mu}=\left(\begin{smallmatrix}
D_\mu & 0& 0\\
0 & e_{\mu}^{X}& 0\\
0 & 0 & e_{\mu}^{U}
\end{smallmatrix}\right),
\; \; \; \; \mathbb{B}=\left(\begin{smallmatrix}
B& 0 & 0 \\
0 & \beta^{X} & 0\\
0 & 0 & \beta^{U}
\end{smallmatrix}\right).
\end{align}
We know from \cite[Sect. 3]{HIR} that $ \beta^{X},\beta^{U} $ are admissible control operators for $ Q^X,Q^U $, respectively, and by assumption \textbf{(A3)} $ B $ is admissible for $ A $. Thus the operator $ \mathbb{B} $ is an admissible control operator for $ \mathcal{A} $. Define the operator
\begin{align*}
\mathcal{C}:=\mathcal{M}_{\vert_{D(\mathcal{A})}}.
\end{align*}
We know from \cite[Theorem 3]{HIR} that $ D, K_{1} $ are admissible observation operators for $ Q^X,Q^U $, respectively,  and by assumption \textbf{(A3)} $ C $ is admissible for $ A $. It follows that the operator $ \mathcal{C} $ is an admissible observation operator for $ \mathcal{A} $. We now prove that the triple $ (\mathcal{A},\mathbb{B},\mathcal{C}) $ is regular with identity operator $ I_{\mathcal{U}} $ as an admissible
feedback. According to \cite[Theorem 4.2.1]{Sta}, the control maps associated to $ \mathbb{B} $ are given by
$$
\mathrm{\Phi}_{t}
\left(\begin{smallmatrix}
v \\u\\u'
\end{smallmatrix}\right)=
\left(\begin{smallmatrix}
\int_0^t T_{-1}(t-s)Bv(s)ds\\
\int_0^t S_{-1}^{X}(t-s)\beta^{X}u(s)ds\\
\int_0^t S_{-1}^{U}(t-s)\beta^{U}u'(s)ds
\end{smallmatrix}\right),\; \; \; t\geq 0,\; \;
\left(\begin{smallmatrix}
v \\u\\u'
\end{smallmatrix}\right)\in L^{p}([0,+\infty);\mathcal{U}).
$$
Let denote by $ D_{\Lambda},K_{1,\Lambda},C_{\Lambda} $ the Yosida extension of $ D, K_1 , C $ with respect to $ Q^{X},Q^{U} $ and $ A $, respectively. Then, the Yosida extension of $ \mathcal{C} $ with respect $ \mathcal{A} $ is given by
\begin{align*}
\mathcal{C}_{\Lambda} =\left(\begin{smallmatrix}
C_{\Lambda} & 0 & 0 \\
I & D_{\Lambda} & K_{1,\Lambda}\\
0 & 0 & 0
\end{smallmatrix}\right), \;\; D(\mathcal{C}_{\Lambda})=D(\mathcal{C}_{\Lambda})\times D( D_{\Lambda}) \times D(K_{1,\Lambda}).
\end{align*}
Moreover, according to \cite[Theorem 3]{HIR}, the triples $ (Q^{X},\beta^{X},D) $ and $ (Q^{U},\beta^{U},K_1 ) $ are regular state-space operators. This with the assumption \textbf{(A3)} implies that $ \text{Rang }\mathrm{\Phi}_{t} \subset D(\mathcal{C}_{\Lambda})$ for a.e. $ t\geq 0 $, cf. \cite[Theorem 5.6.5]{Sta}. We then select
\begin{align*}
\left(\mathbb{F}_{\tau}\left(\begin{smallmatrix}
v \\u\\u'
\end{smallmatrix}\right)\right)(t):=\mathcal{C}_{\Lambda}\mathrm{\Phi}_{t}\left(\begin{smallmatrix}
v\\u\\u'
\end{smallmatrix}\right),\; \; t\in [0,\tau],\; \left(\begin{smallmatrix}
v \\u\\u'
\end{smallmatrix}\right)\in L^{p}([0,+\infty);\mathcal{U}).
\end{align*}
With this it is not difficult to see that the triple $ (\mathcal{A},\mathbb{B},\mathcal{C} )$ is well-posed state-space operators. Moreover, as $R(\mu,\mathcal{A}_{-1})\mathbb{B}=\mathcal{D}_\mu$ for $\mu\in\rho(A),$ we have 
\begin{align*}
{\rm Range}\left(R(\mu,\mathcal{A}_{-1})\mathbb{B}\right)\subset D(C_{\Lambda})\times D( D_{\Lambda}) \times D(K_{1,\Lambda})=D(\mathcal{C}_{\Lambda}),
\end{align*}
for $ \mu \in \rho(\mathcal{A}) $ and $ \begin{pmatrix}z& x & x'\end{pmatrix}^{\top} \in \mathcal{U}$, due to \textbf{(A3)} and the fact that the triples $ (Q^{X},\beta^{X},D) $ and $ (Q^{U},\beta^{U},K_1 ) $ are regular state-space operators as shown in \cite[Theorem 3]{HIR}. Hence $ (\mathcal{A},\mathbf{B},\mathcal{C} )$ is regular regular state-space operators on $ \mathcal{U},\mathcal{X},\mathcal{U} $. We now prove that identity operator $ I_\mathcal{U} $ is an admissible feedback. Clearly, we can write
\begin{align*}
\mathbb{F}_{\tau}=\left(\begin{smallmatrix}
\mathbb{F}_{\tau}^{A,C}& 0 & 0\\
\Phi_{t}^{A} & \mathbb{F}_{\tau}^{Q^{X},D}& \mathbb{F}_{\tau}^{Q^{U},K_1}\\
0 & 0 & 0
\end{smallmatrix}\right)
\end{align*}
with
\begin{align*}
 (\mathbb{F}_{\tau}^{Q^{X},D}u)(t)&=D_{\Lambda}\Phi_{t}^{X} u,\;\; \;\; (\mathbb{F}_{\tau}^{A,C}v)(t)=C_{\Lambda}\Phi_{t}^{A}v
,\\(\mathbb{F}_{\tau}^{Q^{U},K_1}u')(t)&=K_{1,\Lambda}\Phi_{t}^{U} u',
\end{align*}
for $ t\in[0,\tau] $ and $ \left(\begin{smallmatrix}
v\\u\\u'
\end{smallmatrix}\right)\in L^{p}([0,+\infty);\mathcal{U}) $. Thus, $ I_{\mathcal{U}}-\mathbb{F}_{\tau_0} $ is invertible in $ L^{p}([0,\tau_0];\mathcal{U})  $, due to assumption \textbf{(A3)} and the fact that the triples $ (Q^{X},\beta^{X},D) $ is regular state-space operators with $ I_X $ as admissible feedback, cf. \cite{HIR}. Hence, by \cite[Theorem 4.1]{HMR} the operator $\mathcal{A}_{G,M} $ generates a strongly continuous semigroup on $ \mathcal{X} $. This ends the proof.
\end{proof}

The fact that the perturbed Cauchy problem \eqref{S3.1.Sy1} is well-posed follows from the following result:
\begin{theorem}\label{S3.T1}
The operator $ \mathfrak{A}:=\mathcal{A}_{G,M}+\calP $ generates a strongly continuous semigroup $ (\mathfrak{U}(t))_{t\geq 0} $ on $ \mathcal{X} $ satisfying $ \mathfrak{U}(t)\zeta\in D(\calC_{\Lambda})\cap D(\mathbb{P}_{\Lambda}) $ for all $ \zeta \in \calX $ and and almost every $ t\geq 0 $ and
\begin{align*}
\mathfrak{U}(t)&=\calT_{G,M}(t)(t) +\int_{0}^{t} \calT_{G,M}(t)(t-s)\calP \mathfrak{U}(s) ds,\; \; \; \text{on }D(\mathcal{A}_{G,M}),\\
\calT_{G,M}(t)&=\calT(t)+\int_{0}^{t}\calT_{-1}(t-s)\mathbb{B}\calC_{\Lambda}\calT_{G,M}(s)ds,\; \; \;\; \; \; \; \;\; \text{on } \calX.
\end{align*}
\end{theorem}
\begin{proof}
In order to prove the claim of the above theorem, we first define
\begin{align*}
\mathbb{P}:=\calP_{\vert_{D(\calA)}}.
\end{align*}
By Proposition \ref{S3.P1} and \cite[Theorem 3.14]{EN}, it suffices to check that $ \mathbb{P} $ is a Miyadera-Voigt perturbation for $ \mathcal{A}_{G,M} $. To this end, we only need to show that $ \mathbb{P} $ is an admissible observation operator for $ \mathcal{A}_{G,M} $ for a certain $ 1<p<\infty $. In fact, by Proposition \ref{S3.P1} and \cite[Theorem 4.1]{HMR} the semigroup generated by $  \mathcal{A}_{G,M} $ on $ \calX $ is given by
\begin{align}\label{S3.2}
\calT_{G,M}(t)=\calT(t)+\int_{0}^{t}\calT_{-1}(t-s)\mathbb{B}\calC_{\Lambda}\calT_{G,M}(s)ds,\; \; \text{on } \calX,
\end{align}
for any $ t\geq 0 $. Moreover, for any $ \t>0 $ there exist $ \gamma:=\gamma(\t)>0 $ such that
\begin{align}\label{dad}
\int_{0}^{\tau} \Vert \calC_{\Lambda}\calT_{G,M}(s)\zeta\Vert^{p}ds\leq \gamma^{p}\Vert \zeta\Vert^{p},
\end{align}
for any $ \zeta\in \calX $. Since  $ D, K_{1} $ are admissible observation operators for $ Q^X,Q^U $, respectively, cf \cite[Theorem 3]{HIR}, it follows that $ \mathbb{P} $ is admissible observation operator for $ \calA $. Moreover, The Yosida extension of $\mathbb{P} $ with respect to $ \calA $ is given by
\begin{align*}
\mathbb{P}_{\Lambda} =\left(\begin{smallmatrix}
0& L_{\Lambda} & B_{1,\Lambda}\\
0 & 0 & 0  \\
0 & 0 & 0
\end{smallmatrix}\right), \;\; \mathbb{P}_{\Lambda}=X\times D(L_{\Lambda})\times D( B_{1,\Lambda}),
\end{align*}
where $ L_{\Lambda},B_{1,\Lambda} $ denotes the Yosida extension of $ L, B_1 $ with respect to $ Q^{X},Q^{U} $, respectively. According to \cite[Lemma 3.6]{HMR}, we have
$$ L_{\Lambda_{\vert_{W^{1,p}([-r,0],X)}}}=L \; \; \text{ and } \; \; B_{1,\Lambda_{\vert_{W^{1,p}([-r,0],U)}}}=B_1$$
Thus
\begin{align} \label{S3.3}
\mathbb{P}_{\Lambda}=\calP,\; \; \text{on } X\times W^{1,p}([-r,0],X)\times  W^{1,p}([-r,0],U).
\end{align}
On the other hand, by the same argument as in the proof of Proposition \ref{S3.P1} the triple  $ (\calA,\mathbb{B},\mathbb{P}) $ is regular state-space operators. In particular the input-output operators, cf. \cite[Theorem 2.7]{Sta}, satisfy
$$
\int_{0}^{\tau} \Big\Vert \mathbb{P}_{\Lambda}\int_{0}^{t}  \calT_{-1}(t-s)\mathbb{B}\xi(s)ds\Big\Vert^{p}dt\leq \delta^{p}\Vert \xi\Vert^{p}_{L^{p}[0,\t];\calU)},
$$
for all $ \t>0 $ and all input $ \xi\in L^{p}[0,\t];\calU) $. In particular for $ \xi(s)= \calC_{\Lambda}\calT_{\calG,\calM}(s)\zeta$ and using \eqref{dad}, we have
\begin{align}\label{dad1}
\int_{0}^{\tau} \Big\Vert \mathbb{P}_{\Lambda}\int_{0}^{t}  \calT_{-1}(t-s)\mathbb{B}\calC_{\Lambda}\calT_{G,M}(s)\zeta ds\Big\Vert^{p}dt\leq (\gamma\delta)^{p}\Vert \zeta\Vert^{p},
\end{align}
for any $ \zeta\in D(\mathcal{A}_{G,M}) $. Now using \eqref{S3.2}, \eqref{S3.3}, \eqref{dad1} and the fact that $ \calP $ is admissible for $ \calA $, we obtain that $ \calP $ is admissible for $ \mathcal{A}_{G,M} $. So by \cite[Theorem 2.1]{Hadd}, the operator $ ( \mathcal{A}_{G,M}+\calP,D(\mathcal{A}_{G,M})) $ generates a strongly continuous semigroup on $ \mathcal{X} $. Therefore, the statements follows from \cite[Theorem 4.1]{HMR}.
\end{proof}

Next we compute the spectrum of the generator $ \mathfrak{A} $. To this end, we need the
following remarks.
\begin{remark}\label{S3.R1}
It is to be noted that, according to \cite[Theorem 4.1]{HMR}, the operator
\begin{align*}
\calQ_{D}\varphi = Q_{m}^{X}, \quad D(\calQ_{D})= \left\lbrace \varphi \in W^{1,p}([-r,0],X): \quad\varphi(0)=\int_{-r}^{0}d\eta(\theta)\varphi(\theta) \right\rbrace .
\end{align*}
generates a strongly continuous semigroup on $ L^{p}([-r,0],X) $. Moreover, for $ \mu\in \rho(\calQ_{D}) $ (or equivalently $ 1\in \rho(De_\mu^{X}) $), we have
\begin{align*}
R(\mu,\calQ_{D})=\left(I+e_{\mu}^{X}(I-D e_{\mu}^{X})^{-1}D\right)R(\mu,Q^{X}).
\end{align*}
\end{remark}
\begin{lemma}\label{S3.R2}
Let the assumptions of Proposition \ref{S3.P1} be satisfied. Then, for $\mu \in \rho(A)  $, we have
\begin{align*}
\mu\in \rho(\mathcal{A}_{G,M}) \Leftrightarrow 1\in \rho(MD_\mu)\cap\rho(De_\mu)  .
\end{align*}
In this case,
\begin{align}\label{S3.4}
R(\mu,\mathcal{A}_{G,M})&=\left(I+\calD_{\mu}(I-\calM \calD_{\mu})^{-1}\calM\right)R(\mu,\calA)\nonumber\\
&=\left(\begin{smallmatrix}
R(\mu,A_{G,M}) & 0 & 0 \\
e_{\mu}^{X}(I-De_{\mu}^{X})^{-1}R(\mu,A_{G,M}) & R(\mu,\calQ_D) & e_{\mu}^{X}(I-De_{\mu}^{X})^{-1}K_1 R(\mu,Q^{U}) \\
0 & 0 &  R(\mu,Q^{U})
\end{smallmatrix}\right).
\end{align}
\end{lemma}
\begin{proof}
Under the assumptions \textbf{(A1)}-\textbf{(A3)} and according to \cite[Theorem 4.1]{HMR} we have
\begin{align*}
A_{G,M}=A_{m},\; \; D(A_{G,M} )=\{z\in Z:\: Gz=Mz\}
\end{align*}
generates a strongly continuous semigroup $ (T_{G,M}(t))_{t\geq 0} $ on $ X $. Moreover, for $\mu \in \rho(A)  $, we have
\begin{align*}
\mu\in \rho(A_{G,M}) \Leftrightarrow 1\in \rho(MD_\mu).
\end{align*}
In this case
$$ R(\mu,A_{G,M})=\big(I+D_{\mu}(I_{\partial X}-M D_{\mu})^{-1}M\big)R(\mu,A) $$
On the other hand, for $\mu \in \rho(A)  $ we define
\begin{align*}
I_{\mathcal{U}}-\mathcal{M}\calD_{\mu} = \begin{pmatrix}
I_{\partial X}-MD_{\mu}& 0 & 0\\
-D_\mu & I_{X}-De^{X}_\mu & -K_1e^{U}_\mu\\
0 & 0 & I_X
\end{pmatrix}.
\end{align*}
Thus $ I_{\mathcal{U}}-\mathcal{M}\calD_{\mu} $ is invertible if and only if $ 1\in \rho(MD_\mu)\cap\rho(De_\mu)   $, hence $ \mu\in \rho(\mathcal{A}_{G,M}) $ is equivalent to the fact that $ 1\in \rho(MD_\mu)\cap\rho(De_\mu)   $, according to  \cite[Theorem 4.1]{HMR}. Moreover, \eqref{S3.4} is a direct computation using the expression
$$ R(\mu,\mathcal{A}_{G,M})=\left(I+\calD_{\mu}(I-\calM \calD_{\mu})^{-1}\calM\right)R(\mu,\calA) .$$
\end{proof}

We are now in a position to rigourously characterize the spectrum of the generator $ \mathfrak{A} $.
\begin{proposition}\label{S3.P2}
Let the assumptions of Proposition \ref{S3.P1} be satisfied and let $ \mu\in \rho(A) $. Then
\begin{align*}
\mu\in \rho(\mathfrak{A}) \Leftrightarrow 1\in \rho(\Delta(\mu))\Leftrightarrow 1\in \rho(MD_\mu)\cap\rho(De_\mu^{X}),
\end{align*}
where $ \Delta(\mu):=e_{\mu}^{X}(I-De_{\mu}^{X})^{-1}R(\mu,A_{G,M})L$. In addition,
\begin{align}\label{S3.resol}
R(\mu,\mathfrak{A})=\left(\begin{smallmatrix}
R(\mu,A_{G,M})[I+L\Gamma(\mu)R(\mu,A_{G,M})]& R(\mu,A_{G,M})LR(1,\Delta(\mu))R(\mu,\calQ_{D}) & \Lambda (\mu)R(\mu,Q^{U})\\
\Gamma(\mu)R(\mu,A_{G,M}) & R(1,\Delta(\mu))R(\mu,\calQ_{D})&\Omega (\mu)R(\mu,Q^{U})\\
0 & 0 & R(\mu,Q^{U})
\end{smallmatrix}\right),
\end{align}
for $1\in \rho(MD_\mu)\cap\rho(De_\mu^{X})   $. Where the operator $ \calQ_D $ as in Remark \ref{S3.R1} and
\begin{align*}
\Gamma(\mu) &:=R(1,\Delta(\mu))e_{\mu}^{X}(I-De_{\mu}^{X})^{-1},\;\;
\Omega (\mu):=\Gamma(\mu) \Big(K_{1}+R(\mu,A_{G,M}) B_{1}\Big)\\
\Lambda (\mu)&:= R(\mu,A_{G,M})\Big(L\Omega (\mu)+B_1\Big) .
\end{align*}
\end{proposition}
\begin{proof}
Let $ \mu\in \rho(A)\cap \rho(A_{G,M}) $ and $ (x,f,g)^{\top}\in \calX $, we are seeking for $ (z,\varphi,\psi)^{\top}\in D(\calA_{G,M}) $ such that
\begin{align}\label{S3.5}
(\mu-\mathfrak{A})(z,\varphi,\psi)^{\top}=(x,f,g)^{\top}.
\end{align}
From the proof of the Theorem \ref{S3.T1} (in particular \eqref{S3.3}), we have seen that on $ D(\calA_{G,M}) $ we can write $ \mathfrak{A}=\calA_{G,M}+\calP $. Thus, according to Lemma \ref{S3.R2}, for $ \mu \in \rho(A_{G,M})\cap \rho(\calA_{G,M}) $ we have
\begin{align}\label{S3.6}
 \mu-\mathfrak{A}&=\mu-\calA_{G,M}-\calP\nonumber\\
 &=(\mu-\calA_{G,M})(I-R(\mu,\calA_{G,M})\calP)\\
  &=(\mu-\calA_{\calG,\calM})\left(\begin{smallmatrix}
I & -R(\mu,A_{G,M})L & -R(\mu,A_{G,M})B_1 \\
0 & I- e_{\mu}^{X}(I-De_{\mu}^{X})^{-1}R(\mu,A_{G,M})L& -e_{\mu}^{X}(I-De_{\mu}^{X})^{-1}R(\mu,A_{G,M})B_1 \nonumber\\
0 & 0 &  I
\end{smallmatrix}\right).
\end{align}
So, the above matrix together with the equation \eqref{S3.5} yields
\begin{align}\label{S3.7}
z-R(\mu,A_{G,M})L\varphi-R(\mu,A_{G,M})B_1\psi=&x \nonumber\\
 (I- e_{\mu}^{X}(I-De_{\mu}^{X})^{-1}R(\mu,A_{G,M})L)\varphi-e_{\mu}^{X}(I-De_{\mu}^{X})^{-1}R(\mu,A_{G,M})B_1\psi=&f\\
 \psi=&g \nonumber.
\end{align}
Assume that $ 1\in \rho(\Delta(\mu)) $. Then,
$$
\varphi=R(1,\Delta(\mu))f+R(1,\Delta(\mu))e_{\mu}^{X}(I-De_{\mu}^{X})^{-1} R(\mu,A_{G,M})B_{1}g.
$$
On the other hand, replacing the value of $ \varphi $ in \eqref{S3.7},
\begin{align*}
z&=x+R(\mu,A_{G,M})LR(1,\Delta(\mu))f\\
&+R(\mu,A_{G,M})\big( LR(1,\Delta(\mu))e_{\mu}^{X}(I-De_{\mu}^{X})^{-1} R(\mu,A_{G,M})B_{1}+B_1 \big)g.
\end{align*}
Thus $ \mu\in \rho(\mathfrak{A}) $ and \eqref{S3.resol} holds.
\end{proof}

\section{ Frequency domain characterization for approximate controllability of neutral delay systems}\label{sec:3}
In this section, conditions for approximate controllability of the abstract perturbed boundary control systems of neutral type \eqref{Sy.1} will be proposed. Indeed, by using the feedback theory of regular linear systems and methods of functional analysis, necessary and sufficient conditions of approximate controllability are formulated and proved.
\subsection{Abstract setting}\label{S4.Sub1}
We write below system \eqref{Sy.1} as an abstract perturbed boundary control system with input space. For this put
$$ \mathcal{B}\left(\begin{smallmatrix}
v \\
u
\end{smallmatrix}\right)=\left(\begin{smallmatrix}
B_{0}u \\
0 \\
0
\end{smallmatrix}\right),\quad\mathcal{K}\left(\begin{smallmatrix}
v \\
u
\end{smallmatrix}\right)=\left(\begin{smallmatrix}
Kv \\
K_{0}u\\
0
\end{smallmatrix}\right),\; \; u,v\in U.$$
With the above notations and according to Section \ref{sec:2} we can reformulate the abstract boundary control system of neutral type \eqref{Sy.1} as
\begin{align}\label{Pe}
\left\lbrace
\begin{array}{lll}
\dot{\zeta}(t)&=&[\mathcal{A}_{m}+\calP]\zeta(t)+\mathcal{B}u(t),\quad t \geq 0,\\ \zeta(0)&=&(\varrho_0,\varphi,\psi)^{\top}, \\
\calG \zeta(t)&=&\calM \zeta(t)+\mathcal{K}u(t),\qquad\quad\;\; t \geq 0.
\end{array}
\right.
\end{align}
Moreover, combining Theorem \ref{S3.T1} with \cite[Theorem 4.3]{HMR}, it shows that the system \eqref{Pe}, hence the system \eqref{Sy.1}, has a unique solution. This solution coincides with the solution of the open loop system
\begin{align}\label{S3.Sy2}
\left\lbrace
\begin{array}{lll}
\dot{\zeta}(t)&=&\mathfrak{A}_{-1}\zeta(t)+(\mathbb{B}\calK+\mathcal{B})\left(\begin{smallmatrix}
v(t)\\ u(t)
\end{smallmatrix}\right) \quad t> 0,\\
\zeta(0)&=&(\varrho_0,\varphi,\psi)^{\top}.
\end{array}
\right.
\end{align}
In particular, according to \cite[Theorem 4.3]{HMR} the state trajectory of the system \eqref{S3.Sy2}, for the initial state $ \zeta(0) $, is given by
\begin{align}\label{S3.8}
\zeta(t)=\mathfrak{U}(t)\zeta(0)+\displaystyle\int_{0}^{t}\mathfrak{U}_{-1}(t-s)(\mathbb{B}\calK+\mathcal{B})\left(\begin{smallmatrix}
v(s)\\ u(s)
\end{smallmatrix}\right)ds,\; \; t\geq 0.
\end{align}
\begin{remark}\label{S3.R3}
Observe that the function $ \zeta(.) $ given by \eqref{S3.8} is only a strong solution of \eqref{S3.Sy2}, which is defined for any $\zeta(0)\in \calX$ and $ u,v\in L^{p}_{loc}([0,\infty);U) $. For classical solutions, new extrapolation spaces associated with $\calX$ and more conditions on the smoothness of the control function $u $ are needed.
\end{remark}
To state our results on approximate controllability of \eqref{Pe}, we need first to define the concept of approximate controllability for \eqref{S3.Sy2}.
\begin{definition}\label{S3.D1}
According to the equation \eqref{S3.8}, we define the operator
$$ \Phi^{\mathfrak{A}}(t)u:=\int_{0}^{t}\mathfrak{U}_{-1}(t-s)(\mathbb{B}\calK+\mathcal{B})\left(\begin{smallmatrix}
v(s)\\ u(s)
\end{smallmatrix}\right)ds, $$
for $ t\geq 0 $ and $ u,v\in L^{p}([0,t];U) $. The system \eqref{Pe} is said to be $ \mathfrak{X}$-approximately controllable if
\begin{align*}
Cl\left(\bigcup_{t\geqslant 0}P_{\mathfrak{X}}(\mathcal{R}(t))\right) =\mathfrak{X},
\end{align*}
where $ \mathfrak{X} $ can be any of $ \mathcal{X}$, $X$, $X\times L^{p}([-r,0];X) $, $L^{p}([-r,0];X)  $ and  $ L^{p}([-r,0];U)  $, and $ P_{\mathfrak{X}} $ is the projection operator from $ \mathcal{X} $ to $ \mathfrak{X} $. In particular, $ P_{\mathcal{X}} =I $.
\end{definition}
\begin{remark}\label{S3.R4}
Notice that when the system \eqref{Pe} is $ X$-approximately controllable, it means that the corresponding state $z$ is approximately controllable; when it is $ X\times L^{p}([-r,0];X) $-approximately controllable, it means that the corresponding state $z$ and $ x_{t} $ are approximately controllable; when it is $L^{p}([-r,0];U) $-approximately controllable, it means that the corresponding state $u_{t}$ is approximately controllable. Apparently, it is always $L^{p}([-r,0];U) $-approximately controllable.
\end{remark}
In the following we will use the duality of product spaces. Assume that $ X $ and $ U $ are reflexive Banach spaces (or more generally its satisfies the Radon-Nikodym property) and $ 1<p<\infty $. Then the dual space $ \calX'$ of $ \calX $ is identified with the product space $ X'\times L^{q}([-r,0],X')\times L^{q}([-r,0],U')$ with $ q $ satisfying $ \tfrac{1}{p}+\tfrac{1}{q}=1 $. the For the sake of simplicity,

Then the following theorem holds.
\begin{theorem}\label{T.3}
Under the framework of Definition \ref{S3.D1} and with the notation in Proposition \ref{S3.P2}, the following assertions are equivalent:
\begin{itemize}
\item[\emph{(a)}] the system \eqref{Pe} is $ X\times L^{p}([-r,0],X)$-approximately controllable,
\item[\emph{(b)}] for $ 1\in \rho(MD_\mu)\cap\rho(De_\mu^{X})  $, $ x'\in X' $ and $ \varphi\in L^{q}([-r,0],X') $, the fact that
\begin{align*}
&\Big\langle [I+R(\mu,A_{G,M})L\Gamma(\mu)]R(\mu,A_{G,M})\big(BKv+B_{0}u\big)\\
& + R(\mu,A_{G,M})L\Gamma(\mu)K_{0}u+ \Lambda (\mu)e_\mu^{U} u,x'\Big\rangle\\
 &+\Big\langle\Gamma(\mu) R(\mu,A_{G,M})\big(BKv+B_{0}u\big) + \Gamma(\mu)K_{0}u+\Omega (\mu)e_{\mu}^{U}u,\varphi \Big\rangle =0,\;\forall u,v\in U,
\end{align*}
implies that $ x'=0 $ and $ \varphi =0 $.
\end{itemize}
In this case, there exists $ \t>0 $ such that the system \eqref{Pe} is approximately controllable.
\end{theorem}
\begin{proof}
Let $ \mu\in \rho(A) $ such that $ 1\in \rho(MD_\mu)\cap\rho(De_\mu)\cap \rho(\Delta(\mu)) $. According to Proposition \ref{S3.P2}, there is
\begin{align*}
R(\mu,\mathfrak{A})=\left(\begin{smallmatrix}
R(\mu,A_{G,M})[I+L\Gamma(\mu)R(\mu,A_{G,M})]& R(\mu,A_{G,M})LR(1,\Delta(\mu))R(\mu,\calQ_{D}) & \Lambda (\mu)R(\mu,Q^{U})\\
\Gamma(\mu)R(\mu,A_{G,M}) & R(1,\Delta(\mu))R(\mu,\calQ_{D})&\Omega (\mu)R(\mu,Q^{U})\\
0 & 0 & R(\mu,Q^{U})
\end{smallmatrix}\right)
\end{align*}
On the other hand, we have
\begin{align*}
(\mathbb{B}\calK+\calB)\left(\begin{smallmatrix}
v\\u
\end{smallmatrix}\right)=\left(\begin{smallmatrix}
BKv\\
\beta^{X}K_{0}u\\
\beta^{U}u
\end{smallmatrix}\right)+\left(\begin{smallmatrix}
B_{0}u \\
0 \\
0
\end{smallmatrix}\right)=\left(\begin{smallmatrix}
BKv+B_{0}u \\
\beta^{X}K_{0}u\\
\beta^{U}u
\end{smallmatrix}\right).
\end{align*}
Thus,

$(\mathbb{B}\calK+\calB)\left(\begin{smallmatrix}
v\\u
\end{smallmatrix}\right)=$
\begin{align*}
&\left(\begin{smallmatrix}
R(\mu,A_{G,M})[I+L\Gamma(\mu)R(\mu,A_{G,M})]& R(\mu,A_{G,M})LR(1,\Delta(\mu))R(\mu,\calQ_{D}) & \Lambda (\mu)R(\mu,Q^{U})\\
\Gamma(\mu)R(\mu,A_{G,M}) & R(1,\Delta(\mu))R(\mu,\calQ_{D})&\Omega (\mu)R(\mu,Q^{U})\\
0 & 0 & R(\mu,Q^{U})
\end{smallmatrix}\right)
\left(\begin{smallmatrix}
BKv+B_{0}u \\
\beta^{X}K_{0}u\\
\beta^{U}u
\end{smallmatrix}\right)\\
&=\left(\begin{smallmatrix}
[I+R(\mu,A_{G,M})L\Gamma(\mu)]R(\mu,A_{G,M})\big(BKv+B_{0}u\big) + R(\mu,A_{G,M})L\Gamma(\mu)K_{0}u+ \Lambda (\mu)e_\mu^{U} u\\\\
\Gamma(\mu) R(\mu,A_{G,M})\big(BKv+B_{0}u\big) + \Gamma(\mu)K_{0}u+\Omega (\mu)e_{\mu}^{U}u\\\\
e_\mu^{U} u
\end{smallmatrix}\right),
\end{align*}
where we have used the fact that $ R(\mu,\calQ_D)= (I-e_{\mu}^{X}D)^{-1}R(\mu,Q^{X}) $ and $ e_\mu^{X}(I-De_{\mu}^{X})^{-1}= (I-e_{\mu}^{X}D)^{-1} e_\mu^{X}$. Now using the same strategy as in \cite[Proposition 3]{EHR1}, we calim that (\emph{a})$\Leftrightarrow $(\emph{b}).

Therefore, according to Remark \ref{S3.R4} the system \eqref{Pe} is always $ L^{p}([-r,0],U)$-approximately controllable. Then, the statement (\emph{a}) yields the existence of a time for which system \eqref{Pe} is approximate controllability.
\end{proof}
\begin{corollary}
According to Theorem \ref{T.3}, the system \eqref{Pe} is $ X $-approximately controllable if and only if, for $ 1\in \rho(MD_\mu)\cap\rho(De_\mu^{X})  $ and $ x'\in X' $, the fact that
	\begin{align*}
	&\Big\langle [I+R(\mu,A_{G,M})L\Gamma(\mu)]R(\mu,A_{G,M})\big(BKv+B_{0}u\big) \\
	&+ R(\mu,A_{G,M})L\Gamma(\mu)K_{0}u+ \Lambda (\mu)e_\mu^{U} u,x'\Big\rangle =0,
	\end{align*}
for all $ v,u\in U $ implies that $ x'=0 $.
\end{corollary}
This corollary characterizes the condition when $z$, partial state of system \eqref{Pe}, can reach all points in $ X $. In general, this is irrelevant to the approximate controllability of system \eqref{Sy.1}.
\begin{corollary}
According to Theorem \ref{T.3}, the system \eqref{Pe} is $ L^{p}([-r,0],X) $-approximately controllable if and only if, for $ 1\in \rho(MD_\mu)\cap\rho(De_\mu^{X})   $ and $ \varphi\in L^{q}([-r,0],X') $, the fact that
	\begin{align}\label{Cond1}
	\Big\langle\Gamma(\mu) R(\mu,A_{G,M})\big(BKv+B_{0}u\big) + \Gamma(\mu)K_{0}u+\Omega (\mu)e_{\mu}^{U}u,\varphi \Big\rangle =0,\;\forall v,u\in U,
	\end{align}
	implies that $ \varphi=0 $.
\end{corollary}

This corollary actually describes the approximate controllability of the system \eqref{Sy.1}.
\begin{remark}\label{S3.R5}
For $ \mu\in \rho(A) $ such that $ 1\in \rho(MD_\mu)\cap\rho(De_\mu^{X})  $, we denote by
$$\Xi(\mu)=(I-De_{\mu}^{X}-R(\mu,A_{G,M})Le_{\mu}^{X})^{-1}.$$
A direct computation, according to Proposition \ref{S3.P2}, it can be shown that
$$\Gamma(\mu):=R(1,\Delta(\mu))e_{\mu}^{X}(I-De_{\mu}^{X})^{-1} =e_{\mu}^{X}\Xi(\mu).$$
\end{remark}
With the notation of the above remark we obtain:
\begin{theorem}\label{T3.4}
The abstract boundary control system of neutral type \eqref{Sy.1} is approximately contrllable if and only if, for $1\in \rho(MD_\mu)\cap\rho(De_\mu^{X}) $ and $ x'\in X' $, the fact that
$$
	\Big\langle \Xi(\mu) \Big(R(\mu,A_{G,M})\big(BKv+B_{0}u+B_{1}e_{\mu}^{U}u\big)+K_{0}u+K_{1}e_{\mu}^{U}u \Big),x'\Big\rangle =0
$$
for all $ v,u\in U $, implies that $ x'=0 $.
\end{theorem}
\begin{proof}
The state $x(.)$ of the system \eqref{Sy.1} is approximately controllable if and only if the system \eqref{Pe} is $ L^{p}([-r,0],X) $-approximately controllable, i.e., the state $ x_{t} $ is approximately controllable. Under the notation of Remark \ref{S3.R5}, the condition \eqref{Cond1} is equivalent to
$$
\Big\langle \Xi(\mu) \Big(R(\mu,A_{G,M})\big(BKv+B_{0}u+B_{1}e_{\mu}^{U}u\big)+K_{0}u+K_{1}e_{\mu}^{U}u\Big),(e_{\mu}^{X})^{\top}\varphi \Big\rangle=0 ,
$$
for all $ u,v\in U $, with $ (e_{\mu}^{X})^{\top}\varphi\in X' $, since $\Omega (\mu)=\Gamma(\mu) [K_{1}+R(\mu,A_{G,M}) B_{1}]$; see Proposition \ref{S3.P2}. Replacing $ (e_{\mu}^{X})^{\top}\varphi $ with $ x '$ results in the conclusion.
\end{proof}
\subsection{A Rank Condition for a Special Case in Control Spaces}\label{S4.Sub2}
In this subsection, we study approximate controllability of neutral delay systems when the control space is finite dimensional. We establish a novel rank condition criteria for approximate controllability of a such class of system.

Before going further and stating the main result of this subsection, we need to introduce some notations. We start with the assumption on the control space $ U=\mathbb{C}^{n}  $. Then, $ K_{0} $ and $ B_{0} $ are finite-rank operators defined by
\begin{align*}
Ku=\displaystyle\sum_{l=1}^{n}K_{l} v_l,\qquad  K_{0}u=\displaystyle\sum_{l=1}^{n}K_{0,l} u_l\qquad \text{and} \qquad B_{0}u=\displaystyle\sum_{l=1}^{n}B_{0,l} u_l,
\end{align*}
where $ K_{0,l},B_{0,l} \in X $. Moreover, denote $ K_{1}$ and $B_{1} $ as
\begin{align*}
K_{1}=\left( K_{1,1},K_{1,2},\ldots ,K_{1,n}\right) \qquad \text{and} \qquad  B_{1}=\left( B_{1,1},B_{1,2},\ldots ,B_{1,n}\right),
\end{align*}
such as $ K_{1,l},B_{1,l} :W^{1,p}([-r,0],\mathbb{C})\longrightarrow X $ for $ l=1,\ldots, n$. Therefore, for $ u=(u_{1},\ldots, u_{n})^{\top}\in \mathbb{C}^{n} $ and $ \mu \in\mathbb{C} $, we have
\begin{align}\label{S3.9}
\begin{split}
K_{1}e_{\mu}^{U}u&=\sum_{l=1}^{n}K_{1,l}e_{\mu}u_{l}=\sum_{l=1}^{n}u_{l}K_{1,l}e_{\mu}1
\\
B_{1}e_{\mu}^{U}u&=\sum_{l=1}^{n}B_{1,l}e_{\mu}u_{l}=\sum_{l=1}^{n}u_{l}B_{1,l}e_{\mu}1
\end{split}
\end{align}
with $ (e_{\mu}1)(\theta)=e^{\mu \theta} $ for $ \theta\in [-r,0] $.

Throughout the following we denote the orthogonal space of a set $ F $ in $ X $ by
\begin{align*}
F^{\bot}=\{x'\in X';\; \langle y,x'\rangle=0,\; \;\forall y\in F\}.
\end{align*}
\begin{remark}\label{S4.R1}
In view of Theorem \ref{T3.4}, the abstract boundary control system of neutral type \eqref{Sy.1} is approximately controllable if and only if, for $1\in \rho(MD_\mu)\cap\rho(De_\mu)\cap \rho(\Delta(\mu)) $,
\begin{align}\label{S3.11}
\resizebox{0.9\hsize}{!}{$\overline{\Big(\Xi(\mu) D_{\mu}(I-MD_\mu)^{-1}K\mathbb{C}^{n}\Big)+\Big(\Xi(\mu) \big(K_{0}+K_{1}e_{\mu}^{U}+R(\mu,A_{G,M})(B_{0}+B_{1}e_{\mu}^{U})\big)\mathbb{C}^{n}\Big)}=X$}.
\end{align}
In fact,
\begin{align*}
R(\mu,A_{G,M})BK &=(I+D_\mu(I-MD_\mu)^{-1}M)R(\mu,A)BK \\
&= (I+D_\mu(I-MD_\mu)^{-1}M)D_\mu K\\
&= D_{\mu}(I-MD_\mu)^{-1}K.
\end{align*}
Moreover, according to \cite[Proposition 2.14.]{Brez}, the above fact is equivalent to that
\begin{align}\label{ker}
\resizebox{0.9\hsize}{!}{$\Big(\Xi(\mu) D_{\mu}(I-MD_\mu)^{-1}K\mathbb{C}^{n}\Big)^{\bot}\cap\Big(\Xi(\mu) \big(K_{0}+K_{1}e_{\mu}^{U}+R(\mu,A_{G,M})(B_{0}+B_{1}e_{\mu}^{U})\big)\mathbb{C}^{n}\Big)^{\bot}=\{0\}$}.
\end{align}
\end{remark}

Next we provide a useful characterization of approximate controllability for system \eqref{Sy.1}. To this end, we denotes by $d_{l}$ the dimension of
$$ \Upsilon_l(\mu):=\Big(\Xi(\mu) D_{\mu}(I-MD_\mu)^{-1}K_{l}\Big)^{\bot},\; \; l=1,\ldots,n, $$
and by $ (\varphi^{1}_{l},\varphi^{2}_{l},\ldots,\varphi^{d_{l}}_{l} )$ the associated basis.
	
In view of Remark \ref{S4.R1}, the statement in Theorem \ref{T3.4} is equivalent to the following theorem:
\begin{theorem}\label{T.4}
Assume that the control space is finite dimensional. Then the neutral delay system \eqref{Sy.1} is approximately controllable if and only if, for $1\in \rho(MD_\mu)\cap\rho(De_\mu^{X})  $,
\begin{align}\label{Rank}
\textbf{Rank}\left(\begin{smallmatrix}
\langle \Xi(\mu)\Pi'_1(\mu)+\Xi(\mu)\Pi_1(\mu)e_{\mu}1,\varphi^{1}_{l}\rangle  &\cdots  &\langle \Xi(\mu)\Pi'_1(\mu)+\Xi(\mu)\Pi_1(\mu)e_{\mu}1,\varphi^{d_{l}}_{l}\rangle\\
	\langle \Xi(\mu)\Pi'_2(\mu)+\Xi(\mu)\Pi_2(\mu)e_{\mu}1,\varphi^{1}_{l}\rangle   &\cdots &  \langle \Xi(\mu)\Pi'_2(\mu)+\Xi(\mu)\Pi_2(\mu)e_{\mu}1,\varphi^{d_{l}}_{l}\rangle \\
	\vdots  & &  \vdots \\
	\langle \Xi(\mu)\Pi'_l(\mu)+\Xi(\mu)\Pi_l(\mu)e_{\mu}1,\varphi^{1}_{l}\rangle  & \ldots & \langle \Xi(\mu)\Pi'_l(\mu)+\Xi(\mu)\Pi_l(\mu)e_{\mu}1,\varphi^{d_{l}}_{l}\rangle
	\end{smallmatrix}\right)=d_l,
\end{align}
for $ l=1,\ldots,n $, where
\begin{align*}
\Pi_l(\mu)&:=R(\mu,A_{G,M})B_{1,l}+K_{1,l} \\
\Pi'_l(\mu)&:=R(\mu,A_{G,M})B_{0,l}+K_{0,l} .
\end{align*}
\end{theorem}
\begin{proof}
First, using the fact that
\begin{align*}
\langle \Xi(\mu)\Pi'_l(\mu)u+\Xi(\mu)\Pi_l(\mu)e_{\mu}u,x'\rangle=\sum_{l=1}^{n} \bar{u}_{l}\big\langle \Xi(\mu)\Pi'_l(\mu)+\Xi(\mu)\Pi_l(\mu)e_{\mu}1,x'\big\rangle,
\end{align*}
we promptly obtain the following:
\begin{align}\label{Ort}
	\resizebox{0.98\hsize}{!}{$\Big(\Xi(\mu) \big(K_{0}+K_{1}e_{\mu}^{U}+R(\mu,A_{G,M})(B_{0}+B_{1}e_{\mu}^{U})\big)\mathbb{C}^{n}\Big)^{\bot}= \Big( \big\{ \Xi(\mu)\Pi'_l(\mu)+\Xi(\mu)\Pi_l(\mu)e_{\mu}1: l=1,\cdots, n\big\}\Big)^{\bot}$}.
	\end{align}
According to Remark \ref{S4.R1}, to prove the claim of the theorem it suffice to prove that the conditions \eqref{ker} and \eqref{Rank} are equivalent. To this end, we denote by $ \mathsf{M}_{l} $ the matrix appearing in \eqref{Rank} and assume that it is not of rank $ d_l $. Then, there exist $ v=(v_{1},\cdots,v_{d_{l}}) \in \mathbb{C}^{d_{l}}\setminus \{0\}$ such that $ \mathsf{M}_{l}v=0 $, i.e,
\begin{align}\label{Ker}
\left(\begin{smallmatrix}
	\sum_{j=1}^{d_{l}}\bar{v}_{j}\left\langle \Xi(\mu)\Pi'_1(\mu)+\Xi(\mu)\Pi_1(\mu)e_{\mu}1,\varphi^{j}_{l}\right\rangle \\
	\sum_{j=1}^{d_{l}}\bar{v}_{j}\left\langle\Xi(\mu)\Pi'_2(\mu)+\Xi(\mu)\Pi_2(\mu)e_{\mu}1,\varphi^{j}_{l}\right\rangle \\
	\vdots \\
	\sum_{j=1}^{d_{l}}\bar{v}_{j}\left\langle\Xi(\mu)\Pi'_l(\mu)+\Xi(\mu)\Pi_l(\mu)e_{\mu}1,\varphi^{j}_{l}\right\rangle
\end{smallmatrix}\right)=\left(\begin{smallmatrix}
	0 \\
	0 \\
	\vdots \\
	0
\end{smallmatrix}\right).
\end{align}	
As $ (\varphi^{1}_{l},\varphi^{2}_{l},\ldots,\varphi^{d_{l}}_{l} )$ is a basis of $\Upsilon_l(\mu)$, we obtain
$$ \sum_{j=1}^{d_{l}}\bar{v}_{j} \varphi^{j}_{l} \in \Upsilon_l(\mu) .$$ On other hand, because of \eqref{Ort} and \eqref{Ker} one can see that
$$ \sum_{j=1}^{d_{l}}\bar{v}_{j} \varphi^{j}_{l}\in \Big(  \Xi(\mu)\Pi'_l(\mu)+\Xi(\mu)\Pi_l(\mu)e_{\mu}1\Big)^{\bot} .$$
This is a contradiction, which show that the condition \eqref{ker} implies \eqref{Rank}. The converse is demonstrated in a similar fashion by using the expression \eqref{Ort} and the basis $ (\varphi^{1}_{l},\varphi^{2}_{l},\ldots,\varphi^{d_{l}}_{l} )$.	The proof is completed.
\end{proof}

\begin{example} Consider the perturbed boundary control time-delay system
	\begin{eqnarray}\label{state-inp}
	\dot{z}(t)&=&A_m z(t)+ P z(t-r)+ N u(t-r) + B u(t),\qquad t\geq 0,\nonumber\\
	Gz(t)&=& Mz(t)+ Kv(t),\qquad t\geq 0,\\
	z(0)&=&z_0, z(\theta)=\varphi, u(\theta)=\psi,\qquad \text{for a.e. }\theta\in[-r,0]\nonumber.
	\end{eqnarray}
Here $A_m:D(A_m)\subset X\longrightarrow X$ is a closed, linear differential operator on a reflexive Banach space $ X $ and $ G,M: D(A_m)\longrightarrow\partial X$ are unbounded trace operators. $ B,N:\mathbb{C}^{n} \longrightarrow X$ are linear bounded operators and $ K:\mathbb{C}^{n} \longrightarrow \partial X $. This system can be obtained from the system \eqref{Sy.1} by imposing $ D=0 $, $ K_{0}=0 $, $ K_{1}=0 $, $ L=P\delta_{-r} $, and $ B_{1}=N\delta_{-r} $, where $ \delta_{-r} $ is the Dirac operator. Thus, according to Theorem \ref{S3.T1}, the following holds.
\begin{corollary}
Let $ A_m,G,M $ be satisfies the assumptions \textbf{(A1)}-\textbf{(A3)}. Then the perturbed boundary control time-delay system \eqref{state-inp} is wellposed.
\end{corollary}
In this case, for $\mu \in \rho(A)\cap\rho(A_{G,M}+e^{-r\mu}M)  $, we have
$$
	\Xi(\mu)=(\mu I-A_{G,M}-e^{-r\mu}M)^{-1},
$$
and
\begin{align*}
Kv&=\sum_{i=1}^{n}K_{i}v^i,\; K_{i}\in X \qquad	Bu=\sum_{i=1}^{n}B_{i}u^i,\; B_{i}\in X
\\
Nu_t&=\sum_{i=1}^{n}N_{i}u_t^i,\; N_{i}\in X.
\end{align*}
\begin{corollary}
According to Theorem \ref{T.4}, the system \eqref{state-inp} is approximately controllable if and only if, for $\mu \in \rho(A)\cap\rho(A_{G,M}+e^{-r\mu}M)  $,
\begin{align*}
\textbf{Rank}\left(\begin{smallmatrix}
	\langle \Xi(\mu)(N_1+B_{1} e^{-r\mu}),\psi^{1}_{i} \rangle & \langle \Xi(\mu)(N_1+B_{1} e^{-r\mu}),\psi^{2}_{i} \rangle &\cdots  &\langle \Xi(\mu)(N_1+B_{1} e^{-r\mu}),\psi^{d_{i}}_{i} \rangle\\
	\langle \Xi(\mu)(N_2+B_{2} e^{-r\mu}),\psi^{1}_{i}\rangle & \langle \Xi(\mu)(N_2+B_{2} e^{-r\mu}),\psi^{2}_{i}\rangle &\cdots & \langle \Xi(\mu)(N_2+B_{2} e^{-r\mu}),\psi^{d_{i}}_{i} \rangle\\
	\vdots & & \vdots & \vdots\\
	\langle \Xi(\mu)(N_l+B_{l} e^{-r\mu}),\psi^{1}_{i}\rangle & \langle \Xi(\mu)(N_l+B_{l} e^{-r\mu}),\psi^{2}_{i}\rangle &\cdots & \langle \Xi(\mu)(N_l+B_{k} e^{-r\mu}),\psi^{d_{i}}_{i} \rangle
	\end{smallmatrix}\right)=d_{i}.
\end{align*}
Where $ d_{i}=\text{ dim }\big(\Xi(\mu) D_{\mu}(I-MD_\mu)^{-1}K_{i}\big)^{\bot} $ and $ (\psi^{1}_{i},\psi^{2}_{i},\ldots,\psi^{d_{i}}_{i} )$ is the associated basis, for $ i=1,2, \ldots,n$.
\end{corollary}
Note that the equation \eqref{state-inp} is a slight generalization of the one proposed in \cite[Theorem 4.2.6]{CZ}, where the system operator $ A_m $ and delay operators are just constant matrices. Thus, the result from the above corollary is more general.
\end{example}

\section{Application: Transport network system of neutral type and input delays}\label{sec:5}
Let us consider a finite, connected graph $ \mathsf{G}=(\mathsf{V},\mathsf{E}) $ and a flow on it (the latter is described by the differential equation \eqref{S5.1} below). The graph $ \mathsf{G} $ is composed by $ n\in \mathbb{N} $ vertices $ \alpha_{1},\,\ldots,\alpha_n $, and by $ m\in \mathbb{N} $ edges $ e_{1},\,\ldots,e_m $ which are assumed to be normalized on the interval $ [0,1] $. We shall denote the vertices at the endpoints of the edge $ e_{j} $ by $e_{j}(1)$ and $e_{j}(0)$, respectively, and assume that the particles flows from $e_{j}(1)$ to $e_{j}(0)$.

This section characterizes approximate controllability of the following system of
transport network system of neutral type and input delays:
\begin{align}\label{S5.1}
\begin{cases}
     \dfrac{\partial }{\partial t}\varrho^j(t,x)
    =c^j(x)\dfrac{\partial }{\partial x}\varrho^j(t,x)+q^j(x)\varrho^j(t,x) +\displaystyle\sum_{k=1}^{m}L_{jk}z^{k}(t+\cdot,\cdot),\; x\in (0,1), t\geq 0,\\
\varrho^j(0,x)= g^{j}(x),\;\; x\in (0,1), \\
      \mathsf{i}^{-}_{ij}c^{j}(1)\varrho^j(t,1)=
 \mathsf{w}_{ij}^{-}\displaystyle\sum_{k=1}^{m} \mathsf{i}^{+}_{ik}c^{k}(0)\varrho^j(t,0)
 +\sum_{l=1}^{n_0}\mathrm{k}_{il}v^{l}(t),\qquad\qquad\qquad t\geq 0,\\
      z^j(\theta,x)=\varphi^{j}(\theta,x),\; u^{j}(\theta)=\psi^{j}(\theta) ,\qquad\qquad\quad\qquad\qquad\qquad\;\;\theta\in[-r,0],x\in (0,1),\\
      \varrho^j(t,x)=\left[z^{j}(t,x)- \displaystyle\sum_{k=1}^{m}D_{jk}z^{k}(t+\cdot,\cdot)-\displaystyle\sum_{i=1}^{n}\mathsf{k}_{ij}u^{j}(t+\cdot)-\mathrm{b}_{ij}u^{j}(t)\right]
     \end{cases}
\end{align}
for $ i=1,\ldots,n$ and $ j=1,\ldots,m$ with $n,\,m\in \N$. Here, the coefficients $ \textsf{i}^{-}_{ij} $ and $ \textsf{i}^{+}_{ij} $ are the entries of the so-called the outgoing and incoming incidence matrix of $ G $ (denoted by $\mathcal{I}^{+}$ and $\mathcal{I}^{-}$), respectively, defined as
\begin{align*}
\textsf{i}^{-}_{ij}:=
\begin{cases}
1,\quad \text{if  } v_{i}=e_{j}(1),
\\
0,\quad \text{otherwise}.
\end{cases},
\textsf{i}^{+}_{ij}:=
\begin{cases}
1, \quad \text{if  } v_{i}=e_{j}(0),
\\
0, \quad \text{otherwise}.
\end{cases}
\end{align*}
The coefficients $ 0\leqslant\textsf{w}^{-}_{ij}$ determine the proportion of mass leaving vertex $ v_i $ into the edge $ e_j $ and define a graph matrix called the weighted outgoing incidence matrix of $ G $, denoted by $\mathcal{I}_{w}^{-}$. Moreover, we impose \emph{the Kirchhoff condition}
\begin{align}\label{kirch}
\sum_{j=1}^{m}\textsf{w}^{-}_{ij}=1,\; \forall i=1,\ldots,n.
\end{align}
Let $ X=L^{p}([0,1];\mathbb{C}^{m}) $, $ \partial X=\mathbb{C}^{n} $ and define the operator $ A_m $ as
$$(A_{m}g)^j(x):=c^{j}(x)\frac{d}{dx}g^j(x)+q^{j}(x).g^j(x) $$
with domain
\begin{align*}
g\in D(A_{m}):=\left\{ g =(g^{1},\ldots,g^{m})\in( W^{1,p}[0,1])^{m}:  g(1)\in \text{Rang}(I^{-}_{w})^{\top}\right\}.
\end{align*}
Moreover, we define the boundary operators $ G,M: D(A_m)\longrightarrow \partial X$ by
\begin{align}\label{S5.2}
Gf:=f(1)
, \qquad
Mf:=c^{-1}(1)\mathbb{B}c(0)f(0).
\end{align}
Clearly, $ G $ satisfies the assumptions \textbf{(A1)} and \textbf{(A2)}. Therefore, according to \cite[Theorem 3.6]{EHR2}, it follows that:
\begin{lemma}
Define the operators
\begin{align*}
A:=(A_{m})_{\vert _{\ker G}}, \;\;\;B=(\mu -A_{-1})G^{-1}_{ \vert_{\ker ( \mu -A_{m} )}}, \;\;\; C=M_{\vert_{D(A)}}, \; \; \;  \mu \in\mathbb{C}.
\end{align*}
Then the triple $ (A,B,C) $ satisfy the assumption \textbf{(A3)}.
\end{lemma}
So, we have.
\begin{lemma}
The operator
\begin{align}
\calA:=A_m,\; \; \; D(\calA)=\left\{f\in W^{1,p}([0,1],\mathbb{C}^{m}):f(1)=c^{-1}(1)\mathbb{B}c(0)f(0)\right\}.
\end{align}
generates a strongly continuous semigroups $ (T(t))_{t\geqslant 0} $ on $ X $, where $ \mathbb{B}:=(\mathcal{I}_{w}^{-})^{\top}\mathcal{I}^{+} $ is the weighted (transposed) adjacency matrix of the line graph (i.e., the graph obtained from $ \mathsf{G} $ by exchanging the role of the vertices and edges).
\end{lemma}
\begin{proof}
For the proof of this result we refer to \cite[Theorem 3.6]{EHR2}.
\end{proof}

Moreover, we obtain.
\begin{corollary}
For $ \mu \in \rho(\calA)$, we have
\begin{align*}
R(\mu ,\mathcal{A})=(I+D_{\mu}(I_{\mathbb{C}^{n}}-\mathbb{A}_{\mu})^{-1}M)R(\mu,A),
\end{align*}
where
\begin{align*}
 (D_{\mu}v)(x)&=\text{diag}\left(e^{\xi^{j}(x,1)-\mu \tau^{j}(x,1)}\right) v,\; \; Mg:=c(1)^{-1}\mathbb{B}c(0)g(0) \\
 (R(\mu,A)f)^{j}(x)&=\int_{x}^{1} e^{\xi^{j}(x,y)-\mu \tau^{j}(x,y)}c^j(y)f^{j}(y)dy
\end{align*}
for $ g\in D(A_m) $, $v\in \mathbb{C}^n,\, f\in \mathbb{C}^m ,\, x\in[0,1]$ and
\begin{align*}
(\mathbb{A}_{\mu})_{ip}=
\begin{cases}
\mathsf{w}_{pj}^{-}e^{\xi_{j}(0,1)-\mu \tau_{j}(0,1)},& \text{if  } v_{i}=e_{j}(0) \text{ and }
 v_{p}=e_{j}(1),
\\
0, &\text{otherwise}.
\end{cases}
\end{align*}
\end{corollary}
\begin{proof}
A proof of this lemma can be found in \cite[Corollary 3.8]{EHR2}.
\end{proof}
On the other hand, in order to apply the results of the previous sections, let us assume that
\begin{align*}
D_{jk}(g^{k})&=\int_{-r}^{0}d\eta_{jk}(\theta)g^k(\theta),\; L_{jk}(g^{k})=\int_{-r}^{0}d\gamma_{jk}(\theta)g^k(\theta),\\
 \mathsf{k}_{ij}(f^{j})&=\int_{-r}^{0}d\vartheta_{ij}(\theta)f^j(\theta),
\end{align*}
for $ g\in W^{1,p}([-r,0],X) $ and $ f\in W^{1,p}([-r,0],\mathbb{C}^{n}) $, where $ \eta,\gamma:[-r,0]\longrightarrow \mathcal{L}(X) $ and $\vartheta:[-r,0]\longrightarrow \mathcal{L}(\C^n,X) $ are functions of bounded variations continuous at zero with $ \eta(0)=\gamma(0)=\vartheta(0)=0 $. Thus the system \eqref{S5.1} is rewritten in the form \eqref{Sy.1} with $ D=(D_{jk})_{m\times m} $, $ L=(L_{jk})_{m\times m} $, $ K_{0}=(\mathsf{b}_{ij})_{m\times n} $, $ K_{1}=(\mathsf{k}_{ij})_{m\times n} $ and $ B_{1}=B_0\equiv 0 $.

Therefore, according to Theorem \ref{S3.T1}, the transport network system of neutral type \eqref{S5.1} is well-posed.
 \begin{corollary}
The operator $ (\mathfrak{A},D(\mathfrak{A})) $ defined by
\begin{align*}
\mathfrak{A}&=\left(\begin{smallmatrix}
\calA & L & 0\\
0 & Q_{m}^{X} & 0\\
0 & 0 & Q_{m}^{U}
\end{smallmatrix}\right),\\
 D(\mathfrak{A})&=\left\{\left(\begin{smallmatrix}
 f \\ \varphi \\ \psi
\end{smallmatrix}\right)\in D(A_{m})\times W^{1,p}([-r,0],X)\times D(Q^{U}):\; \left(\begin{smallmatrix}
f(1)=c^{-1}(1)\mathbb{B}c(0)f(0) \\ \varphi(0)=f+D\varphi+K_1 \psi
\end{smallmatrix}\right)\right\}.
\end{align*}
generates a strongly continuous semigroups $ (\mathfrak{U}(t))_{t\geqslant 0} $ on $ X\times L^{p}([-r,0];X)\times  L^{p}([-r,0];\mathbb{C}^{n_0}) $.
\end{corollary}
In this case, for $1\in \rho(\mathbb{A}_\mu)\cap\rho(De_\mu) $, we have
$$\Xi(\mu)=(I-De_{\mu}-R(\mu,A)Le_{\mu}-D_{\mu}(I_{\mathbb{C}^{n}}-\mathbb{A}_{\mu})^{-1}MR(\mu,A)Le_{\mu})^{-1}.$$
The fact that the transport network system of neutral type \eqref{S5.1} is approximately controllable follows from the following result:
\begin{corollary}
Let $1\in \rho(\mathbb{A}_\mu)\cap\rho(De_\mu) $, then the system \eqref{S5.1} is approximately controllable  if and only if the following matrix is of rank $d_i$
\begin{align*}
\sum_{j=1}^{m}\left(\begin{smallmatrix}
\big\langle \Xi(\mu)\big(\mathrm{b}_{1j}+\mathsf{k}_{1j}e_\mu1 \big),\varphi^{1}_{i}\big\rangle  &\cdots  & \big\langle \Xi(\mu)\big(\mathrm{b}_{1j}+\mathsf{k}_{1j}e_\mu1\big),\varphi^{d_{i}}_{i}\big\rangle\\
	\big\langle \Xi(\mu)\big(\mathrm{b}_{2j}+\mathsf{k}_{2j}e_\mu1\big),\varphi^{1}_{i}\big\rangle   &\cdots &  \big\langle \Xi(\mu)\big(\mathrm{b}_{2j}+\mathsf{k}_{2j}e_\mu1\big),\varphi^{d_{i}}_{i}\big\rangle \\
	\vdots  & &  \vdots \\
	\big\langle \Xi(\mu)\big(\mathrm{b}_{nj}+\mathsf{k}_{nj}e_\mu1\big),\varphi^{1}_{i}\big\rangle  & \ldots & \langle \Xi(\mu)\big(\mathrm{b}_{nj}+\mathsf{k}_{nj}e_\mu1\big),\varphi^{d_{i}}_{i}\big\rangle,
	\end{smallmatrix}\right), \; \; i=1,\ldots,n.
\end{align*}
Here $ d_{i}$ denote the dimension of $\left(\displaystyle\sum_{l=1}^{n_0}\Xi(\mu) D_{\mu}(I_{\mathbb{C}^{n}}-\mathbb{A}_{\mu})^{-1}\mathrm{k}_{il}\right)^{\bot} $ and $ (\varphi^{1}_{i},\varphi^{2}_{i},\ldots,\varphi^{d_{i}}_{i} )$ denote the associated basis, for $ i=1,\ldots,n $.
\end{corollary}
\begin{proof}
The statements follows from Theorem \ref{T.4} with
$$ \Pi_{i}(\mu)=\sum_{j=1}^{m} \Xi(\mu)\mathsf{k}_{ij}, \; \Pi_{i}'\mu)=\sum_{j=1}^{m} \Xi(\mu)\mathrm{b}_{ij},\; \Upsilon_{i}(\mu)= \left(\displaystyle\sum_{l=1}^{n_0}\Xi(\mu) D_{\mu}(I_{\mathbb{C}^{n}}-\mathbb{A}_{\mu})^{-1} \mathrm{k}_{il}\right)^{\bot}.$$
\end{proof}


%
%



\end{document}